\documentclass[oneside,a4paper]{article}
\usepackage{amssymb,amsmath,latexsym,graphicx,tikz,indentfirst,scalerel,hyperref,caption,marginnote,multicol,cite}

\usepackage[english]{babel}
\usepackage{csquotes}

\hypersetup{
  colorlinks   = true, 
  urlcolor     = blue, 
  linkcolor    = blue, 
  citecolor   = red 
}

\usetikzlibrary{automata, positioning, calc,arrows}

\def\dj{d\kern-0.4em\char"16\kern-0.1em}
\def\Dj{\mbox{\raise0.3ex\hbox{-}\kern-0.4em D}}

\newtheorem{theorem}{Theorem}[section]
\newtheorem{lemma}[theorem]{Lemma}
\newtheorem{proposition}[theorem]{Proposition}
\newtheorem{corollary}[theorem]{Corollary}
\newtheorem{definition}[theorem]{Definition}
\newtheorem{example}[theorem]{Example}
\newtheorem{remark}[theorem]{Remark}

\makeatletter
\renewcommand{\@dotsep}{10000}
\makeatother

\newenvironment{proof}
{\noindent
{\it Proof.}}
{\hspace{\stretch{1}}%
$\Box$}

\newcounter{primer}[section]

\DeclareMathOperator{\Cen}{Cen}

\makeatletter
\tikzset{my loop/.style =  {to path={
  \pgfextra{}
  [looseness=4,min distance=5mm]
  \tikz@to@curve@path},font=\sffamily\small
  }}
\makeatletter

\newcommand{\simp}{\stackrel{p}{\sim}}

\newcommand{\sime}{\stackrel{e}{\sim}}

\newcommand{\equive}{\stackrel{e}{\equiv}}

\makeatletter
\newcommand{\myitem}[1]{%
\item[#1]\protected@edef\@currentlabel{#1}%
}
\makeatother

\begin{document}

\thispagestyle{empty}
\begin{center}
\Large{Some New Results Concerning Power Graphs and Enhanced Power Graphs of Groups}
\vspace{3mm}

\normalsize{I. Bo\v snjak, R. Madar\' asz, S. Zahirovi\' c}
\end{center}

\begin{abstract}
The directed power graph $\vec{\mathcal P}(\mathbf G)$ of a group $\mathbf G$ is the simple digraph with vertex set $G$ such that $x\rightarrow y$ if $y$ is a power of $x$. The power graph of $\mathbf G$, denoted by $\mathcal P(\mathbf G)$, is the underlying simple graph. The enhanced power graph $\mathcal P_e(\mathbf G)$ of $\mathbf G$ is the simple graph with vertex set $G$ in which two elements are adjacent if they generate a cyclic subgroup.

In this paper, it is proven that, if two groups have isomorphic power graphs, then they have isomorphic enhanced power graphs, too. It is known that any finite nilpotent group of order divisible by at most two primes has perfect enhanced power graph. We investigated whether the same holds for all finite groups, and we have obtained a negative answer to that question. Further, we proved that, for any $n\geq 0$ and prime numbers $p$ and $q$, every group of order $p^nq$ and $p^2q^2$ has perfect enhanced power graph. We also give a complete characterization of symmetric and alternative groups with perfect enhanced graphs.
%
\end{abstract}

\section{Introduction}

The directed power graph of a group $\mathbf G$ is the simple directed graph whose vertex set is $G$, and in which $x\rightarrow y$ if $y$ is a power of $x$. Its underlying simple graph is called the power graph of the group. The directed power graph was introduced by Kelarev and Quinn \cite{prvi kelarev i kvin}, while the power graph was first studied by Chakrabarty, Ghosh and Sen \cite{cakrabarti}. The power graph has been subject of many papers, including \cite{power graph, power graph 2, cameron3, on the structure, sitov, drugi kelarev i kvin, kelarev3, kelarev4, cameron4,samostalni,cameron-manna-mehatari}. In these papers, combinatorial and algebraic properties of the power graph have received great attention. For more details, papers \cite{pregled za power grafove,cameron-graphs-defined-on-groups,pregledni-cameron-i-druzina} are recommended.

The enhanced power graph of a group is the simple graph whose vertices are elements of the group, and in which two vertices are adjacent if they are powers of some element of the group, i.e. if they generate a cyclic subgroup.  The enhanced power graph was introduced by Aalipour et al. \cite{on the structure}, although the complement of this graph, under the name the noncyclic graph of the group, was studied ten years earlier in \cite{abdolahi}. The enhanced power graph was further studied in \cite{nas prvi, bera i bunja, hamzeh2, ma i she, panda dalal i kumar, cameron-phan, parveen-kumar-singh-ma, parveen-dalal-kumar, parveen-kumar}. In these papers, combinatorial properties of the enhanced power graph, its relation to the power graph of the group, as well as  algebraic properties of its automorphism group, have received significant attention.



Cameron \cite{power graph 2} proved that two finite groups that have isomorphic power graphs also have isomorphic directed power graphs. Bo\v snjak, Madar\' asz and Zahirovi\' c \cite{nas prvi} proved that the enhanced power graph of a finite group determines the directed power graph. Therefore, by \cite{power graph 2} and \cite{nas prvi}, in the case of finite groups, the power graph, the directed power graph and the enhanced power graph determine each other. Cameron, Guerra and Jurina \cite{cameron3} proved that, if two torsion-free groups of nilpotency class $2$ have isomorphic power graphs, then they have isomorphic directed power graphs as well. Zahirovi\' c \cite{samostalni} proved that this implication holds whenever at least one of the two groups is torsion-free and of nilpotency class $2$, and in \cite{drugi samostalni} he generalized this result by proving the same implication when at least one of the groups does not contain a Pr\" ufer group as a subgroup. By the result from \cite{drugi samostalni}, if a group does not contain a Pr\" ufer group as a subgroup, then its enhanced power graph is determined by the power graph. In Section \ref{sekcija medjusobno odredjivanje opsteg i nenula-stepenog grafa}, we are going to prove that the power graph of any group determines the enhanced power graph.

Many combinatorial properties of the power graph have been studied, and all that has motivated similar research for the enhanced power graph. Aalipour et al. \cite{on the structure} showed that the power graph of every group of bounded exponent is perfect. They also proved that the clique number of the power graph of any group is at most countable, and they posed the question whether the chromatic number of the power graph of any group is at most countable too. Shitov \cite{sitov} gave the affirmative answer to their question by proving that every power-associative groupoid, i.e. groupoid whose all one-generated subgroupoids are semigroups, has power graph whose chromatic number is at most countable. Aalipour at al. \cite{on the structure} proved that the clique number of the enhanced power graph is at most countable. Recently, Cameron and Phan \cite{cameron-phan} proved that the enhanced power graph of every finite group is weakly perfect. In \cite{nas prvi}, the authors proved that a finite nilpotent group has perfect enhanced power graph if and only if it has at most two noncyclic Sylow subgroups.
Consequently, for any prime numbers $p$ and $q$, every nilpotent group of order $p^nq^m$ has perfect enhanced power graph. In Section \ref{sekcija za perfektnost} of the present paper, we show that this result cannot be generalized for all finite groups, by presenting an example of a group of order $1296=2^43^4$ whose enhanced power graph is not perfect. We prove that any group of order $p^nq$ or $p^2q^2$ has perfect enhanced power graph. Further in this section, a characterization of finite symmetric and alternative groups with perfect enhanced power graph is given. Namely, the enhanced power graph of $\mathbf S_n$ is perfect if and only if $n\leq 7$, and the enhanced power graph of $\mathbf A_n$ is perfect if and only if $n\leq 8$.

Note that this version of the manuscript differs from the first version submitted to arXiv. Namely, we removed one theorem and its consequence which turned out to be incorrect. For detailed information, see comments at the arXiv page.

\section{Basic Notions and Notations}

In this paper, by a graph we mean a simple graph. For a graph $\Gamma$, $V(\Gamma)$ and $E(\Gamma)$ denote the vertex set and the edge set of $\Gamma$, respectively. If two vertices $x$ and $y$ of $\Gamma$ are adjacent, we denote that fact by $x\sim_\Gamma y$, or shortly by $x\sim y$. We write $x\simeq_\Gamma y$ if $x\sim_\Gamma y$ or $x=y$.  The {\bf closed neighborhood} of a vertex $x$ of a graph $\Gamma$ is the set \[\overline N_\Gamma(x)=\{y\mid y\sim_\Gamma x\text{ or }y=x\},\] and we shall shortly denote it by $\overline N(x)$. If two vertices $x$ and $y$ of $\Gamma$ have the same closed neighborhood, we shall denote it by $x\equiv_\Gamma y$, or simply by $x\equiv y$. Subgraph of $\Gamma$ induced by a set $X\subseteq V(\Gamma)$ we denote by $\Gamma[X]$. The {\bf complement} of a graph $\Gamma$ is the graph $\overline \Gamma$ with vertex set $V(\Gamma)$ and such that $x\sim_{\overline \Gamma} y$ if and only if $x\not\sim_{\Gamma}y$.

{\bf Eccentricity} of a vertex is the maximal distance between that vertex and any other vertex of the graph. {\bf Radius} of a graph is the minimal eccentricity of a vertex of the graph. The {\bf center} of a graph is the set of all of its vertices with the minimal eccentricity.

For graphs $\Gamma$ and $\Delta$, the {\bf strong product} of $\Gamma$ and $\Delta$ is the graph $\Gamma\boxtimes\Delta$ with vertex set $V(\Gamma)\times V(\Delta)$ such that
\begin{align*}
(x_1,y_1)\sim_{\Gamma\boxtimes\Delta}(x_2,y_2)\text{ if }(x_1=x_2\vee x_1\sim_\Gamma x_2)\wedge(y_1= y_2\vee y_1\sim_\Delta y_2)
\end{align*}
whenever $(x_1,y_1)\neq(x_2,y_2)$.

The vertex set of a directed graph, or a digraph, $\vec\Gamma$ is denoted by $V(\vec{\Gamma})$. Its edge set, which consists  of ordered pairs of different vertices of $\vec{\Gamma}$, is denoted by $E(\vec\Gamma)$. Its edges are also called directed edges or arcs. If $(x,y)\in E(\vec\Gamma)$, we denote that fact by $x\rightarrow_\Gamma y$, or shortly by $x\rightarrow y$.

All through this paper, algebraic structures such as groups will be denoted by bold capitals, and their universes will be denoted by respective regular capital letters. For elements $x$ and $y$ of a group $\mathbf G$, we write $x\approx_{\mathbf G} y$, or simply $x\approx y$, if $\langle x\rangle=\langle y\rangle$, where $\langle x\rangle$ denotes the subgroup generated by $x$. By $o(x)$ we shall denote the order of the element $x$ of a group.

Now we introduce the definitions of the (directed) power graph and the enhanced power graph of a group.

\begin{definition}\label{definicije pridruzenih grafova}
The {\bf directed power graph} of a group $\mathbf G$ is the digraph $\vec{\mathcal P}(\mathbf G)$ whose vertex set is $G$, and in which there is a directed edge from $x$ to $y$, $x\neq y$, if there exists $n\in\mathbb Z\setminus\{0\}$ such that $y=x^n$.

The {\bf power graph} of a group $\mathbf G$ is the graph $\mathcal P(\mathbf G)$ whose vertex set is $G$, and whose vertices $x$ and $y$, $x\neq y$, are adjacent if there exists $n\in\mathbb Z\setminus\{0\}$ such that $y=x^n$ or $x=y^n$.

The {\bf enhanced power graph} of a group $\mathbf G$ is the graph $\mathcal P_e(\mathbf G)$ whose vertex set is $G$, and whose vertices $x$ and $y$, $x\neq y$, are adjacent if there exist $z\in G$ and $n,m\in\mathbb Z$ such that $x=z^n$ and $y=z^m$.
\end{definition}

If there is a directed edge from $x$ to $y$ in $\vec{\mathcal P}(\mathbf G)$, instead of writing $x\rightarrow_{\vec{\mathcal P}(\mathbf G)} y$, we shall denote that by $x\rightarrow_{\mathbf G} y$ or shortly by $x\rightarrow y$. Similarly, instead of writing $x\sim_{\mathcal P(\mathbf G)}y$, we will write $x\simp_{\mathbf G} y$ or shortly $x\simp y$, and instead of writing $x\sim_{\mathcal P_e(\mathbf G)}y$ we shall write $x\sime_{\mathbf G} y$ or shortly $x\sime y$.

If elements $x$ and $y$ of a group $\mathbf G$ have the same closed neighborhood in $\mathcal P(\mathbf G)$, instead of writing $x\equiv_{\mathcal P(\mathbf G)}y$, we shall denote that by $x\equiv_{\mathbf G} y$ or shortly $x\equiv y$. Also, instead of writing $x\equiv_{\mathcal P_e(\mathbf G)}y$, we write $x\equive_{\mathbf G}y$ or shortly $x\equive y$.

We note that the directed power graph and the power graph of a group are often defined such that $x\rightarrow y$ if $y=x^n$ for some $n\in\mathbb Z$, and such that $x\simp y$ if $y=x^n$ or $x=y^n$ for some $n\in\mathbb Z$. Here, we use slightly different definition of the power graph by which the identity element of the group is adjacent to no element of infinite order. This way it will be more convenient to state some of our arguments, and it is justified because, by \cite[Theorem 1]{samostalni}, the power graphs by these two definitions determine each other up to isomorphism.

\section{The Power Graph and the Enhanced Power Graph of a Group}\label{sekcija medjusobno odredjivanje opsteg i nenula-stepenog grafa}

In this section, we prove that the enhanced power graph of a group is determined by the power graph. The converse is not true; there are groups with isomorphic enhanced power graphs whose power graphs are not isomorphic. For example, the enhanced power graphs of $\mathbb Z$ and $\mathbf C_{p^\infty}$ are both complete graphs, while their power graphs are not isomorphic ($\mathcal P(\mathbf C_{p^{\infty}})$ is complete, while $\mathcal P(\mathbb Z)$ is not).

We shall start by introducing the notions of the infinite-order segment and the finite-order segment of the power graph of a group $\mathbf G$. For a group $\mathbf G$, by $G_\infty$ and $G_{<\infty}$ we denote the set of all finite-order elements and the set of all infinite-order elements of $\mathbf G$, respectively. The subgraph of $\mathcal P(\mathbf G)$ induced by $G_{<\infty}$ is called the {\bf finite-order segment} of the power graph of $\mathbf G$. The subgraph of $\mathcal P(\mathbf G)$ induced by $G_{\infty}$ is called the {\bf infinite-order segment} of $\mathcal P(\mathbf G)$. Note that the finite-order segment of the power graph of a group is a connected component of $\mathcal P(\mathbf G)$, while the infinite-order segment is a union of connected components of  $\mathcal P(\mathbf G)$. The finite-order segment and the infinite-order segment of the directed power graph and the enhanced power graph of a group are introduced in the same way. Also, since the enhanced power graph is connected, neither the infinite-order segment nor the finite-order segment of the enhanced power graph is a connected component.

Notice that the radius of the finite-order segment of the power graph of any group is $1$ because the identity element is adjacent to all other elements of finite order of the group. Therefore, for a group $\mathbf G$ and the finite-order segment $\Phi$ of its power graph, the {\bf center} of $\Phi$ is the set
\[\Cen(\Phi)=\big\{x\in G_{<\infty}\mathrel{\big|} x\simp y\text{ for all }y\in G_{<\infty}\setminus\{x\}\big\}.\]

By the following lemma, which was proven in \cite{drugi samostalni}, an isomorphism between the power graphs of two groups maps the finite-order segment onto the finite-order segment.

\begin{lemma}[{\cite[Lemma 4]{drugi samostalni}}]\label{konacni se slikaju na konacne}
Let $\mathbf G$ and $\mathbf H$ be groups, and let $\varphi:G\rightarrow H$ be an isomorphism from $\mathcal P(\mathbf G)$ to $\mathcal P(\mathbf H)$. Then $\varphi(G_{<\infty})=H_{<\infty}$.
\end{lemma}

\begin{proposition}\label{izomorfizam izmedju infinite-order segmenata}
Let $\mathbf G$ and $\mathbf H$ be groups with isomorphic power graphs. Then the infinite-order segments of the enhanced power graphs of $\mathbf G$ and $\mathbf H$ are isomorphic too.
\end{proposition}

\begin{proof}
By Lemma \ref{konacni se slikaju na konacne}, $\mathcal P(\mathbf G)\cong\mathcal P(\mathbf H)$ implies that $\mathcal P(\mathbf G)$ and $\mathcal P(\mathbf H)$ have isomorphic finite-order segments and isomorphic infinite-order segments. By \cite[Theorem 10]{drugi samostalni}, if the power graphs of two groups have isomorphic infinite-order segments, then their directed power graphs have isomorphic infinite-order segments. Since the directed power graph of a group determines the enhanced power graph, an isomorphism between infinite-order segments of directed power graphs of the groups is also an isomorphism between the infinite-order segments of enhanced power graphs. Therefore, $\mathcal P_e(\mathbf G)$ and $\mathcal P_e(\mathbf H)$ have isomorphic infinite-order segments.
\end{proof}\\

By the above proposition, it remains to prove that an isomorphism between the finite-order segments of the power graphs of two groups is an isomorphism between the finite-order segments of the enhanced power graphs. The next proposition was proven in \cite{drugi samostalni}, and it is a generalization of \cite[Proposition 4]{power graph 2} proven by Cameron.

\begin{proposition}[{\cite[Proposition 13]{drugi samostalni}}]\label{kameronova propozicija}
Let $\mathbf G$ be a group, let $\Phi$ denote the finite-order segment of $\mathcal G^\pm(\mathbf G)$, and let $\lvert\Cen(\Phi)\rvert>1$. Let $S=\Cen(\Phi)$. Then one of the following holds:
\begin{enumerate}
\item $\mathbf G_{<\infty}$ is a Pr\" ufer group. In this case, $S=G_{<\infty}$ and $S$ is infinite.
\item $\mathbf G_{<\infty}$ is a cyclic group of prime power order. In this case, $S=G_{<\infty}$ and $S$ is finite.
\item $\mathbf G_{<\infty}$ is a cyclic group whose order is the product of two different prime numbers. In this case, $\lvert S\rvert\geq \frac 12\lvert G_{<\infty}\rvert$ and the set $G_{<\infty}\setminus S$ induces a disconnected subgraph of $\Phi$.
\item $\mathbf G_{<\infty}$ is a cyclic group whose order is divisible by at least two different prime numbers, but whose order is not the product of two different prime numbers. In this case, the set $G_{<\infty}\setminus S$ induces a connected subgraph of $\Phi$.
\item There is a prime number $p$ such that the order of every element from $G_{<\infty}$ is a power of $p$, but $\langle G_{<\infty}\rangle$ is not a cyclic, nor a Pr\" ufer group. In this case, $\lvert S\rvert<\frac 12\lvert G_{<\infty}\rvert$ and the set $G_{<\infty}\setminus S$ induces a disconnected subgraph of $\Phi$.
\end{enumerate}
\end{proposition}

By the above proposition we will be able to deal with the finite-order segments by separating the proof into two cases: one in which the center of the finite-order segment of the power graph is trivial, i.e. contains only one element, and the other in which the center of the finite order segment contains more than one element. First we shall focus on the first case.



If $x$ and $y$ are adjacent in the power graph, there is a way to determine whether $x\rightarrow y$ or $y\rightarrow x$ by observing cardinalities of $[x]_\approx$ and $[y]_\approx$ and their relation in the power graph; however, from the power graph of $\mathbf G$ we \enquote{do not see} $\approx$-classes. Nevertheless, we do see $\equiv$-classes, and each $\equiv$-class is a union of $\approx$-classes. Cameron \cite{power graph 2} noticed that in the power graph of a finite group, all $\equiv$-classes that are unions of more than one $\approx$-classes are of one specific form. The following lemma, which was proven in \cite{drugi samostalni}, is a generalization of that.

\begin{lemma}[{\cite[Lemma 15]{drugi samostalni}}]\label{dva tipa klasa}
Let $\mathbf G$ be a group such that $\lvert\Cen(\Phi)\rvert=1$, where $\Phi$ is the finite-order segment of $\mathcal P(\mathbf G)$.  Then every $\equiv$-class $C$ of $\Phi$ is one of the following forms:
\begin{enumerate}
\item $C$ is a $\approx$-class. Such a $\equiv$-class we shall call a {\bf simple $\equiv$-class}.
\item $C=\{x\in\langle y\rangle\mid o(x)\geq p^s\}$, where $p$ is a prime number, $y$ is an element of order $p^r$ for some $r\in\mathbb N$, and where $s\in\mathbb N$ satisfies $r>s>0$. In this case, $C$ is a union of $r-s+1$ $\approx$-classes, and we say that such a $\equiv$-class is a {\bf complex $\equiv$-class}.
\item $C=\bigcup_{k\geq s}[x_k]_\approx$ for some $s\geq 1$, where, for some prime number $p$, each $x_k$ is an element of order $p^k$, and where $x_k\in\langle x_{k+1}\rangle$, for all $k\geq s$. We call such a $\equiv$-class an {\bf infinitely complex} $\equiv$-class.
\end{enumerate}
\end{lemma}


As we will see in Lemma \ref{prepoznavanje tipova klasa}, when the center of the finite-order segment of the power graph of a group is trivial, not just that we see $\equiv$-classes from the power graph, but for each $\equiv$-class we can tell whether it is simple, complex or infinitely complex. Furthermore, by Lemma \ref{prepoznavanje tipova klasa}, an isomorphism between the power graphs of two groups maps every $\equiv$-class onto a $\equiv$-class of the same type. The following lemma was first proven by Cameron \cite{power graph 2} for any finite group whose power graph has trivial center, and in \cite{drugi samostalni} it was proven for any group $\mathbf G$ whose finite-order segment of its power graph has trivial center.

\begin{lemma}[{\cite[Lemma 17]{drugi samostalni}}]\label{prepoznavanje tipova klasa}
Let $\mathbf G$ be a group such that $\lvert\Cen(\Phi)\rvert=1$, where $\Phi$ is the finite-order segment of $\mathcal G^\pm(\mathbf G)$. Let $\hat X$ denote the set $\overline N(X)=\bigcap_{x\in X}\overline N(x)$ for every $X\subseteq G$, and let $C$ be a complex $\equiv$-class. Then the following holds:
\begin{enumerate}
\item $|\hat C|=p^r$ and $|\hat C|-|C|=p^{s-1}$ for some $r,s\in\mathbb N$ such that $r>s> 0$;
\item $C$ is adjacent to no mutually nonadjacent $\equiv$-classes $D$ and $E$ such that $|D|,|E|\leq |C|$.
\end{enumerate}
Further, $p^r$ and $p^s$ are the maximum and the minimum order of an element of $C$, respectively.

If $C$ is a simple $\equiv$-class, then at least one of the above statements is not satisfied.
\end{lemma}

%

Now, we are going to prove the following lemma, which will be useful in the proof of the Theorem \ref{neusmereni odredjuju obogacene kod konacnih}.

\begin{lemma}\label{slozena susedstva imaju sve isto kod oba}
Let $\mathbf G$ be a group such that $\lvert\Cen(\Phi)\rvert=1$, where $\Phi$ is the finite-order segment of $\mathcal P(\mathbf G)$. Let $x$ be an element of $\Phi$ contained in a complex or infinitely complex $\equiv$-class. Then $\overline N_{\mathcal P(\mathbf G)}(x)=\overline N_{\mathcal P_e(\mathbf G)}(x)$.
\end{lemma}

\begin{proof}
By the definition of complex and infinitely complex $\equiv$-classes, the order of $x$ is a power of some prime number $p$. Obviously, for every $y\in G$ such that $x\rightarrow y$, $o(y)$ is a power of $p$ too. Suppose that there is some $y\in G$ such that $y\rightarrow x$ and such that $o(y)$ is not a power of $p$. Then $x^p\simp y^p\not\simp x$, and if there was some $z\in G$ such that $z^p=x$, then $x\simp y\not\simp z$. It follows that $[x]_\equiv$ is a simple $\equiv$-class, which is a contradiction. Therefore, for every $y\in\overline N_{\mathcal P(\mathbf G)}(x)$, $o(y)$ is a power of $p$. Additionally, this also implies that, for any $z\in G$ such that $x\sime z$, $o(z)$ is a power of $p$ as well. Notice that, for any pair of elements $g,h\in G$ whose orders are powers of a prime number, $g\simp h$ if and only if $g\sime h$. Therefore, $x$ has the same closed neighborhood in the power graph and the enhanced power graph of $\mathbf G$.
\end{proof}\\

Now we are ready to prove the main result of this section.

\begin{theorem}\label{neusmereni odredjuju obogacene kod konacnih}
Let $\mathbf G$ and $\mathbf H$ be groups whose power graphs are isomorphic. Then their enhanced power graphs are isomorphic too.
\end{theorem}

\begin{proof}
By Proposition \ref{izomorfizam izmedju infinite-order segmenata}, $\mathcal P_e(\mathbf G)$ and $\mathcal P_e(\mathbf H)$ have isomorphic infinite-order segments. Therefore, it remains to prove that $\mathcal P_e(\mathbf G)$ and $\mathcal P_e(\mathbf H)$ have isomorphic finite-order segments as well. Suppose first that the center of the finite-order segment of $\mathcal P(\mathbf G)$ is nontrivial, i.e. $\lvert\Cen(\mathcal P(\mathbf G))\rvert>1$. Then, by Proposition \ref{kameronova propozicija}, $\mathbf G_{<\infty}$ is a finite cyclic subgroup of $\mathbf G$, or the order of every element of finite order of $\mathbf G$ is a power of a prime number $p$. Moreover, by the same proposition, and because the finite-order segments of $\mathcal P(\mathbf G)$ and $\mathcal P(\mathbf H)$ are isomorphic, $\mathbf G_{<\infty}$ is cyclic if and only if $\mathbf H_{<\infty}$ is cyclic. Also, the order of every element from $G_{<\infty}$ is a power of a prime number if and only if the order of every element from $H_{<\infty}$ is a power of a prime number. In the former case, the finite-order segments of $\mathcal P_e(\mathbf G)$ and $\mathcal P_e(\mathbf H)$ are complete graphs of the same order, and in the latter case, the finite-order segments of $\mathcal P_e(\mathbf G)$ and  $\mathcal P_e(\mathbf H)$ are equal to the finite-order segments of $\mathcal P(\mathbf G)$ and  $\mathcal P(\mathbf H)$, respectively. Therefore, if $\lvert\Cen(\mathcal P(\mathbf G))\rvert>1$, then $\mathcal P_e(\mathbf G)\cong\mathcal P_e(\mathbf H)$.

Suppose now that $\lvert\Cen(\mathcal P(\mathbf G))\rvert=1$. Let $\varphi: G\rightarrow H$ be an isomorphism from $\mathcal P(\mathbf G)$ to $\mathcal P(\mathbf H)$. Then, by Lemma \ref{konacni se slikaju na konacne}, $\varphi|_{G_{<\infty}}$ is an isomorphism between finite-order segments of  $\mathcal P(\mathbf G)$ and $\mathcal P(\mathbf H)$. Let us prove that $\varphi|_{G_{<\infty}}$ is also an isomorphism between the finite-order segments of  $\mathcal P_e(\mathbf G)$ and $\mathcal P_e(\mathbf H)$. Let $x,y\in G_{<\infty}$, and let $x\sime_{\mathbf G} y$. If  $x\simp_{\mathbf G} y$, then obviously $\varphi(x)\sime_{\mathbf H}\varphi(y)$. So suppose that $x\not\simp_{\mathbf G} y$. Then, by Lemma \ref{slozena susedstva imaju sve isto kod oba}, $[x]_{\equiv_{\mathbf G}}$ and $[y]_{\equiv_{\mathbf G}}$ are simple $\equiv_{\mathbf G}$-classes. Furthermore, by Lemma \ref{prepoznavanje tipova klasa},  neither $[x]_{\equiv_{\mathbf G}}$ nor $[y]_{\equiv_{\mathbf G}}$ satisfies both conditions from Lemma \ref{prepoznavanje tipova klasa}. Since $\varphi$ is an isomorphism, neither $[\varphi(x)]_{\equiv_{\mathbf H}}$ nor $[\varphi(y)]_{\equiv_{\mathbf H}}$ satisfies both conditions from Lemma \ref{prepoznavanje tipova klasa}. Therefore, by Lemma \ref{prepoznavanje tipova klasa}, implies that $[\varphi(x)]_{\equiv_{\mathbf H}}$ and $[\varphi(y)]_{\equiv_{\mathbf H}}$ are also simple $\equiv_{\mathbf H}$-classes. Since $x$ and $y$ are not adjacent in $\mathcal P(\mathbf G)$, there is an element $z\in G$ such that $z\rightarrow_{\mathbf G}x,y$. Then $[x]_{\equiv_{\mathbf G}}\neq[y]_{\equiv_{\mathbf G}}\neq[z]_{\equiv_{\mathbf G}}\neq[x]_{\equiv_{\mathbf G}}$, and $[\varphi(x)]_{\equiv_{\mathbf H}}\neq[\varphi(y)]_{\equiv_{\mathbf H}}\neq[\varphi(z)]_{\equiv_{\mathbf H}}\neq[\varphi(x)]_{\equiv_{\mathbf H}}$. Also, $[z]_{\equiv_{\mathbf G}}$ and $[\varphi(z)]_{\equiv_{\mathbf H}}$ are both simple $\equiv$-classes. Now, by Lemma \ref{prepoznavanje tipova klasa}, $z\rightarrow_{\mathbf G}x,y$ implies  $\varphi(z)\rightarrow_{\mathbf H}\varphi(x),\varphi(y)$, i.e. $\varphi(x)\sime_{\mathbf H}\varphi(y)$. In an analogous way it is proven that, for every $x,y\in G$, $\varphi(x)\sime_{\mathbf H}\varphi(y)$ implies $x\sime_{\mathbf G}y$. This proves that $\mathcal P_e(\mathbf G)$ and $\mathcal P_e(\mathbf H)$ have isomorphic finite-order segments.

Let $e_{\mathbf G}$ and $e_{\mathbf H}$ denote the identity elements of $\mathbf G$ and $\mathbf H$, respectively. Now, if $\varphi(e_{\mathbf G})=e_{\mathbf H}$, then $\varphi$ is an isomorphism between $\mathcal P_e(\mathbf G)$ and $\mathcal P_e(\mathbf H)$. On the other hand, if $\varphi(e_{\mathbf G})\neq e_{\mathbf H}$, then $\varphi\circ\tau$ is an isomorphism between $\mathcal P_e(\mathbf G)$ and $\mathcal P_e(\mathbf H)$, where $\tau$ is the transposition of $e_{\mathbf H}$ and $\varphi(e_{\mathbf G})$. Thus, $\mathcal P_e(\mathbf G)$ and $\mathcal P_e(\mathbf H)$ are isomorphic.
\end{proof}

%

\section{Perfectness of the Enhanced Power Graph of a Finite Group}\label{sekcija za perfektnost}


The {\bf chromatic number} of a graph $\Gamma$, denoted by $\chi(\Gamma)$, is the minimum number of colors by which one could color the vertices of $\Gamma$ so that no two adjacent vertices have the same color. A {\bf clique} of a graph $\Gamma$ is a set of its vertices $C$ such that $\Gamma[C]$ is a complete graph. The {\bf clique number} of a graph $\Gamma$, denoted by $\omega(\Gamma)$, is the maximal cardinality of a clique in $\Gamma$. A graph $\Gamma$ is {\bf perfect} if $\omega(\Delta)=\chi(\Delta)$ for every induced subgraph $\Delta$ of $\Gamma$. A graph is a {\bf Berge graph} if neither $\Gamma$ nor the complement of $\Gamma$ contains an odd-length cycle of size at least $5$ as an induced subgraph.

The following theorem, which was proven in \cite{strong perfect graph}, is known as The Strong Perfect Graph Theorem, and it will play an essential role in proving all of the results of this section.

\begin{theorem}[{\cite[Theorem 1.2]{strong perfect graph}}]\label{jaka teorema o perfektnim grafovima}
A finite graph is perfect if and only if it is a Berge graph.
\end{theorem}

Theorem \ref{glavna za nilpotentne i perfektnost} was proven in \cite{nas prvi}. It gives a necessary and sufficient condition for a finite nilpotent group to have perfect enhanced power graph.

\begin{theorem}[{\cite[Theorem 6.2]{nas prvi}}]\label{glavna za nilpotentne i perfektnost}
A finite nilpotent group has perfect enhanced power graph  if and only if  it has at most two noncyclic Sylow subgroups.
\end{theorem}

By Theorem \ref{glavna za nilpotentne i perfektnost}, for any prime numbers $p$ and $q$, the enhanced power graph of any finite nilpotent group of order $p^nq^m$ is perfect. Therefore, it is natural to ask ourselves whether that is true for all finite groups. By the following example however, we give the negative answer to that question.

\begin{example}\label{kontraprimer}
	The enhanced power graph of $\mathbf C_{36}\times\mathbf C_6\times\mathbf S_3$ is not perfect.
\end{example}

\begin{figure}[h!]
	\centering
	\begin{tikzpicture}[>=stealth',shorten >=0pt,auto,node distance=1cm,semithick,on grid,rotate=90]
		\def\nodedistance{1cm}
		
		\tikzstyle{every state}=[minimum size=4pt, fill=black,draw=black,inner sep=0cm]
		\node[state,label=above:{$(e,b^2,e)$}]  (a) at (0:1*\nodedistance) {};
		\node[state,label=left:{$(e,b^3,e)$}]  (b) at (72:1*\nodedistance) {};
		\node[state,label=left:{$(a^4,e,x)$}]  (c)  at (144:1*\nodedistance) {};
		\node[state,label=right:{$(a^6,e,e)$}]  (d) at (216:1*\nodedistance) {};
		\node[state,label=right:{$(a^9,e,y)$}]  (e)  at (288:1*\nodedistance) {};
		
		\path[-] (a) edge[bend right=15]  (b)
		(b)    edge [bend right=15]   (c)
		(c)     edge [bend right=15] (d)
		(d) edge [bend right=15]   (e)
		(e)     edge [bend right=15]   (a);
	\end{tikzpicture}
	\caption{A pentagon in $\mathcal P_e(\mathbf C_{36}\times\mathbf C_6\times\mathbf S_3)$}\label{pentagon u grupi p^nq^m}
\end{figure}

\begin{proof}
	Let $a$ and $b$ be generating elements of $\mathbf C_{36}$ and $\mathbf C_6$, respectively. Let $\mathbf S_3=\langle x,y\rangle$, where $o(x)=3$ and $o(y)=2$, and let $\mathbf G$ denote $\mathbf C_{36}\times\mathbf C_6\times\mathbf S_3$. Let us prove that elements $(a^4,e,x)$, $(a^6,e,e)$, $(a^9,e,y)$, $(e,b^2,e)$, and $(e,b^3,e)$ induce the pentagon in $\mathcal P_e(\mathbf G)$ as in Figure \ref{pentagon u grupi p^nq^m}. Namely, notice that
	\begin{align*}
		(a^4,e,x), (a^6,e,e)&\in\langle (a^2,e,x^2) \rangle, & (e,b^2,e), (e,b^3,e)&\in\langle (e,b,e) \rangle\text{ and}\\
		(a^6,e,e), (a^9,e,y)&\in\langle (a^3,e,y) \rangle, & (e,b^3,e), (a^4,e,x)&\in\langle (a^4,b^3,x) \rangle.\\
		(a^9,e,y), (e,b^2,e)&\in\langle (a^9,b^2,y) \rangle,
	\end{align*}
	Furthermore, by Lemma \ref{deljivost redova} (which shall be proven immediately after this example),  $(a^4,e,x)\not\sime (e,b^2,e)$ since $o\big((e,b^2,e)\big)$ divides $o\big((a^4,e,x)\big)$ and $(e,b^2,e)\not\in\langle(a^4,e,x)\rangle$. Similarly, $(a^6,e,e)\not\sime (e,b^2,e)$ and $(a^6,e,e)\not\sime(e,b^3,e)\not\sime (a^9,e,y)$. Finally, we obtain that $(a^4,e,x)\not\sime (a^9,e,y)$ because they do not commute. This completes our proof.
\end{proof}\\

The following statements will be of use later.

\begin{lemma}\label{deljivost redova}
	Let $\mathbf G$ be a finite group, and let $a$ and $b$ be neighbors in $\mathcal G_e(\mathbf G)$ such that $o(a)$ divides $o(b)$. Then
	\begin{enumerate}
	\item $a$ is a power of $b$.
	\item $\overline N(b)\subseteq \overline N(a)$.
	\end{enumerate}
\end{lemma}

\begin{proof}
	{\it 1.} Since $a$ and $b$ belong to the same cyclic subgroup $C$, the subgroup $\langle a \rangle$ is contained in $C$. Therefore, $\langle a \rangle$ is the only cyclic subgroup of order $o(a)$ which is contained in $C$.
	Since $\langle b \rangle$ is a subgroup of $C$ and contains a subgroup of order $o(a)$, $\langle a \rangle$ must be a subgroup of $\langle b \rangle$.
	
	{\it 2.} If $b$ is adjacent to $c$, they generate a cyclic subgroup which obviously contains $a$. So, $a$ is adjacent to $c$ in the enhanced power graph.
\end{proof}\\

\begin{corollary}\label{deljivost u enhanced grafu}
	Let $\mathbf G$ be a finite group, and let $a$ and $b$ be neighbors in $\mathcal G_e(\mathbf G)$ such that $o(a)$ divides $o(b)$. Then
	\begin{enumerate}
	\item  an induced odd-length cycle in $\mathcal G_e(\mathbf G)$ of size at least $5$ cannot contain both $a$ and $b$.
	\item the complement of an induced odd-length cycle in $\mathcal G_e(\mathbf G)$ of size at least $5$ cannot contain both $a$ and $b$.
	\end{enumerate}
\end{corollary}

\begin{proof}
	Immediate consequence of Lemma \ref{deljivost redova} ({\it 1.}).
\end{proof}\\

\begin{lemma}\label{jedinstvena maksimalna}
	Let $a$ be an element of a finite group $\mathbf G$ which belongs to exactly one maximal cyclic subgroup of  $\mathbf G$. Then $a$ cannot belong neither to an induced odd-length cycle of size at least 5, nor to the complement of an induced odd-length cycle of size at least 5.
\end{lemma}

\begin{proof}
	Let $a$ belongs to a forbidden cycle or its complement. Then $a$ has two neighbors which are not adjacent to each other, so it belongs to two different maximal cyclic subgroups.
\end{proof}\\

\begin{corollary}\label{generatori maksimalne}
	Let $a$ be a generator of a maximal cyclic subgroup of $\mathbf G$. Then $a$ cannot belong neither to an induced odd-length cycle of size at least 5, nor to the complement of an induced odd-length cycle of size at least 5.
\end{corollary}

\begin{proof}
	Immediate consequence of Lemma \ref{jedinstvena maksimalna}.
\end{proof}\\

The following is a consequence of Corollary \ref{deljivost u enhanced grafu}.

\begin{theorem}\label{pnq}
	Let $\mathbf G$ be a finite group of order $p^nq$, for some $n\geq 0$ and some primes $p$ and $q$. Then group $\mathbf G$ has perfect enhanced power graph.
\end{theorem}

\begin{proof}
	The set $G$ can be partitioned into two subsets, $A$ and $B$. $A$ consists of elements of order $p^k$, $k\leq n$, and $B$ consists of elements of order $p^kq$, $k\leq n$. An induced odd-length cycle must contain two adjacent vertices from the same set, which is a contradiction with Corollary \ref{deljivost u enhanced grafu}. The complement of any such cycle contains an odd cycle (not necessarily induced), which causes contradiction again.
\end{proof}\\

\begin{theorem}\label{p2q2}
	A group of order $p^2q^2$, where $p$ and $q$ are distinct primes, has perfect enhanced power graph.
\end{theorem}

\begin{proof}
	Let $\mathbf G$ be a group of order $p^2q^2$. We will assume that $p>q$. Let us suppose that $\mathcal G_e(\mathbf G)$ contains an odd length cycle $K$ of size at least 5 as an induced subgraph. $\mathbf G$ does not have elements of order $p^2q^2$, since it is not cyclic. $K$ cannot contain elements of order $p^2q$
	or $pq^2$, for any of them would generate a maximal cyclic subgroup, which is impossible, according to Corollary \ref{generatori maksimalne}. So, only the elements of order $p$, $q$, $pq$, $p^2$ and $q^2$ could be on $K$. Two elements of order $p^2$ and $q^2$ which belong to $K$ cannot be adjacent, because they would generate a cyclic subgroup of order $p^2q^2$.
	
	$K$ must contain an element of order $p^2$. Otherwise, we would be able to partition vertices of $K$ into two subsets $A$ and $B$, such that $A$ contains elements of order $p$ and $pq$ and $B$ contains elements of order $q$ and $q^2$. According to Corollary \ref{deljivost u enhanced grafu}, no two vertices from $A$ could be adjacent, and no two vertices from $B$ could be adjacent, which contradicts the assumption that $K$ is an odd cycle.
	
	Let $p^2q^2\neq 36$. If $s_p$ denotes the number of Sylow $p$-subgroups, then $s_p=kp+1$ for some nonnegative integer $k$ and  $s_p \mid q^2$. From $p>q$ we conclude that $kp+1=q^2$, or $s_p=1$. If $s_p\neq 1$, we get $kp=(q-1)(q+1)$, which implies $p=q+1$. But this is possible only if $p=3$, $q=2$ and $|G|=36$ (then $s_3=4$). This means that $\mathbf G$ has a unique Sylow $p$-subgroup. As a consequence, we get that no two elements of order $p^2$ could be on $K$, since they would belong to two different Sylow $p$-subgroups. Also, no element of order $p$ could appear on $K$ since every element of order $p$ on $K$ would be contained in another Sylow $p$-subgroup (different from that containing the element of order $p^2$ since those two elements could not be adjacent). Now, we can notice that elements of order $q$ could not belong to $K$, because they can be adjacent only to the element of order $p^2$. The only orders left on the cycle $K$ are $p^2$, $q^2$ and $pq$, but this would contradict to Corollary \ref{deljivost u enhanced grafu}, since elements of order $p^2$ and $q^2$ cannot be adjacent.
	
	Let $p^2q^2=36$. We have already seen that $K$ must contain an element of order 9. This means that 3-Sylow subgroups of $\mathbf G$ are cyclic. If there is more than one of them, then $s_3=4$. Since the cyclic subgroup of order 9 has 6 generators, $\mathbf G$ contains 24 elements of order 9. Every element of order 9 which belong to $K$ must be contained in two cyclic subgroups of order 18, which contain more than 12 elements of order different from 9. Therefore $|G|\geq 37$, which is a contradiction. So, $s_3=1$ and the proof is the same as for the case $p^2q^2\neq 36$.
	
	Suppose now that $\mathcal G_e(\mathbf G)$ contains the complement $\overline K$ of an odd length cycle $K$ of size at least 5 as an induced subgraph. In fact, we may assume that the length of $K$ is at least 7, since the cycle of length 5 is isomorphic to its complement. $\overline K$ cannot contain elements of order $p^2q^2$, $p^2q$ or $pq^2$ for the same reasons that we used to explain why $K$ cannot contain them. Notice that $\overline K$ contains a triangle. By Corollary \ref{deljivost u enhanced grafu}, the orders of the elements in the triangle must be from the set $\{p^2,q^2,pq\}$, but this is impossible, since two elements of order $p^2$ and $q^2$ which belong to $\overline K$ cannot be adjacent.
\end{proof}\\

\begin{remark} In the proof of the previous theorem, we actually proved that if a group $\mathbf G$ of order $p^2q^2$ contains a forbidden cycle, than it has a normal Sylow subgroup. In fact, it is known from group theory that any group of order $p^2q^2$ contains a normal Sylow subgroup.
\end{remark}

In the end, we give a complete characterization of symmetric and alternative groups with perfect enhanced graphs.

\begin{proposition}\label{simetricna}
	The symmetric group $\mathbf S_n$ has perfect enhanced power graph if and only if $n\leq 7$, and the alternative group $\mathbf A_n$ has perfect enhanced power graph if and only if $n\leq 8$.
\end{proposition}

\begin{proof}
	The enhanced power graphs of $\mathbf S_8$ and $\mathbf A_9$ contain a pentagon and a heptagon as induced subgraphs, respectively, as displayed in Figure \ref{pentagon u S8 i A9}.
	Therefore, $\mathcal P_e(\mathbf S_n)$ is not perfect for any $n\geq 8$, and $\mathcal P_e(\mathbf A_n)$  is not perfect for any $n\geq 9$.
	
	\begin{figure}[h!]
		\centering
		\begin{tikzpicture}[>=stealth',shorten >=0pt,auto,node distance=1cm,semithick,on grid,rotate=90]
			\def\nodedistance{1cm}
			
			\tikzstyle{every state}=[minimum size=4pt, fill=black,draw=black,inner sep=0cm]
			\node[state,label=above:$(1\ 2\ 3\ 4\ 5)$]  (a) at (0:1*\nodedistance) {};
			\node[state,label=left:$(6\ 7\ 8)$]  (b) at (72:1*\nodedistance) {};
			\node[state,label=left:$(1\ 2)$]  (c)  at (144:1*\nodedistance) {};
			\node[state,label=right:$(3\ 4\ 5)$]  (d) at (216:1*\nodedistance) {};
			\node[state,label=right:$(6\ 7)$]  (e)  at (288:1*\nodedistance) {};
			
			\path[-] (a) edge[bend right=15]  (b)
			(b)    edge [bend right=15]   (c)
			(c)     edge [bend right=15] (d)
			(d) edge [bend right=15]   (e)
			(e)     edge [bend right=15]   (a);
		\end{tikzpicture}\hspace{1cm}
		%
		%
		\begin{tikzpicture}[>=stealth',shorten >=0pt,auto,node distance=1cm,semithick,on grid,rotate=90]
			\def\nodedistance{1.2cm}
			
			\tikzstyle{every state}=[minimum size=4pt, fill=black,draw=black,inner sep=0cm]
			\node[state,label=above:$(1\ 2\ 3\ 4\ 5)$]  (a) at (0:1*\nodedistance) {};
			\node[state,label=left:$(7\ 8\ 9)$]  (b) at (360/7:1*\nodedistance) {};
			\node[state,label=left:$(1\ 2)(3\ 4)$]  (c)  at (2*360/7:1*\nodedistance) {};
			\node[state,label=left:$(5\ 6\ 7)$]  (d) at (3*360/7:1*\nodedistance) {};
			\node[state,label=right:$(1\ 2)(8\ 9)$]  (e)  at (4*360/7:1*\nodedistance) {};
			\node[state,label=right:$(3\ 4\ 5)$]  (f)  at (5*360/7:1*\nodedistance) {};
			\node[state,label=right:$(6\ 7)(8\ 9)$]  (g)  at (6*360/7:1*\nodedistance) {};
			
			\path[-] (a) edge[bend right=15]  (b)
			(b)    edge [bend right=15]   (c)
			(c)     edge [bend right=15] (d)
			(d) edge [bend right=15]   (e)
			(e)     edge [bend right=15]   (f)
			(f) edge [bend right=15]   (g)
			(g)     edge [bend right=15]   (a);
		\end{tikzpicture}
		\caption{A pentagon in $\mathcal P_e(\mathbf S_8)$ and a heptagon in $\mathcal P_e(\mathbf A_9)$}\label{pentagon u S8 i A9}
	\end{figure}
	
	It is sufficient for us to show that $\mathcal P_e(\mathbf S_7)$ and $\mathcal P_e(\mathbf A_8)$ are perfect. By Theorem \ref{jaka teorema o perfektnim grafovima}, it is sufficient to prove that $\mathcal P_e(\mathbf S_7)$ and $\mathcal P_e(\mathbf A_8)$ do not contain any forbidden graph as an induced subgraph, where, in this proof, by a forbidden graph we mean an odd-length cycle of size at least $5$ or its complement.
	
	The symmetric group $\mathbf S_7$ contains non-identity elements of orders $2$, $3$, $4$, $5$, $6$, $7$, $10$ and $12$. Suppose that $\mathcal G_e(\mathbf S_7)$ contains a forbidden subgraph induced by a set $F$. Then, according to Corollary \ref{generatori maksimalne}, $F$ does not contain any element of order $7$, $10$ or $12$. Every element of order 5 is contained in exactly one maximal cyclic subgroup of $\mathbf S_7$ (of order 10) and it cannot belong to $F$ according to Lemma \ref{jedinstvena maksimalna}. Therefore, only elements of orders $2$, $3$, $4$, and $6$ belong to $F$, but this would contradict Corollary \ref{deljivost u enhanced grafu}. Therefore,  $\mathcal G_e(\mathbf S_7)$ is a Berge graph, and, by Theorem \ref{jaka teorema o perfektnim grafovima}, it is perfect.
	
	Let us prove that $\mathcal P_e(\mathbf A_8)$ is perfect too.  The group $\mathbf A_8$ contains only nonidentity elements of orders $2$, $3$, $4$, $5$, $6$, $7$ and $15$.
	Suppose that $\mathcal G_e(\mathbf A_8)$ has a forbidden subgraph induced by a set $F$. Again, according to Corollary \ref{generatori maksimalne}, $F$ does not contain any element of order $6$, $7$ or $15$. Also, if $g$ has order $5$, then it is contained in exactly one maximal subgroup of $\mathbf A_n$ (of order $15$). Therefore, it cannot belong to $F$. It follows that $F$ contains only elements of orders $2$, $3$, and $4$, which is impossible, due to Corollary \ref{deljivost u enhanced grafu}. Therefore,  $\mathcal G_e(\mathbf A_8)$ is a Berge graph, and, by Theorem \ref{jaka teorema o perfektnim grafovima}, it is perfect.
\end{proof}

\section*{Acknowledgment}

The authors acknowledge financial support of the Ministry of Education, Science and Technological Development of the Republic of Serbia (Grant No. 451-03-68/2020-14/200125).

\vbox{
\vbox{\noindent
Ivica Bo\v snjak

University of Novi Sad, Department of Mathematics and Informatics, Serbia

{\it e-mail}: \href{mailto:ivb@dmi.uns.ac.rs}{ivb@dmi.uns.ac.rs}}\

\vbox{\noindent
Roz\' alia Madar\' asz

University of Novi Sad, Department of Mathematics and Informatics, Serbia

{\it e-mail}: \href{mailto:rozi@dmi.uns.ac.rs}{rozi@dmi.uns.ac.rs}}\

\vbox{\noindent
Samir Zahirovi\' c

University of Novi Sad, Department of Mathematics and Informatics, Serbia

{\it e-mail}: \href{mailto:samir.zahirovic@dmi.uns.ac.rs}{samir.zahirovic@dmi.uns.ac.rs}}}

\end{document}